\documentclass{amsart}
\usepackage{times,amssymb,booktabs,amsmath,mathrsfs,cite}
\usepackage[pagebackref=true]{hyperref}
\renewcommand*{\backref}[1]{}
\renewcommand*{\backrefalt}[4]%
  {{[\tiny\ifcase #1 Not cited.\relax\or Page~#2.\else Pages #2.\fi]}}

\usepackage[svgnames]{xcolor}
\usepackage[hyperpageref]{backref}
\hypersetup{pdftitle={A Specht filtration of an induced Specht module},
  pdfauthor={Andrew Mathas},
  colorlinks = true,
  linkcolor =ForestGreen,
  anchorcolor = red,
  citecolor = blue,
  urlcolor = blue
}

\usepackage[vcentermath]{youngtab}
\def\tab(#1){\mbox{\small$\young(#1)$}\,}

\usepackage{aliascnt}
\def\NewTheorem#1{%
  \newaliascnt{#1}{equation}
  \newtheorem{#1}[#1]{#1}
  \aliascntresetthe{#1}
  \expandafter\def\csname #1autorefname\endcsname{#1}
}

\def\equationautorefname~#1\null{(#1)\null}

\newcounter{main}
\theoremstyle{plain}
\newtheorem{THEOREM}[main]{Theorem}
\swapnumbers
\numberwithin{equation}{section}

\NewTheorem{Proposition}
\NewTheorem{Theorem}
\NewTheorem{Corollary}
\NewTheorem{Lemma}
\NewTheorem{Definition}
\theoremstyle{definition}
\theoremstyle{remark}
\NewTheorem{Remark}

\newaliascnt{Example}{equation}
\aliascntresetthe{Example}

  {\refstepcounter{equation}\trivlist
   \item[\hskip\labelsep\theequation.~\textbf{Example}\space]
   \ignorespaces
  }{\unskip\nobreak\hfil%
    \penalty50\hskip2em\hbox{}\nobreak\hfil$\Diamond$%
    \parfillskip=0pt\finalhyphendemerits=0\penalty-100\endtrivlist
}

\newenvironment{Point}%
  {\refstepcounter{equation}\trivlist
  \item[\hskip\labelsep(\theequation)\space]
   \ignorespaces
  }{\unskip\nobreak\hfil%
    \penalty50\hskip2em\hbox{}\nobreak\hfil$\Diamond$%
    \parfillskip=0pt\finalhyphendemerits=0\penalty-100\endtrivlist
}

\def\map#1#2{\,{:}\,#1\!\longrightarrow\!#2}

{\catcode`\|=\active
  \gdef\set#1{\mathinner{\lbrace\,{\mathcode`\|"8000%
                                   \let|\midvert #1}\,\rbrace}}
}
\def\midvert{\egroup\mid\bgroup}
\let\<\langle
\let\>\rangle

\def\({\Big(}
\def\){\Big)}

\let\gdom=\triangleright

\let\gedom=\trianglerighteq
\let\Sect=\S

\def\N{\mathbb N}
\def\Z{\mathbb Z}

\def\Sym{\mathfrak S}
\def\H{\mathscr{H}}

\def\blam{{\boldsymbol\lambda}}
\def\bmu{{\boldsymbol\mu}}
\def\bnu{{\boldsymbol\nu}}

\def\s{\mathfrak s}
\def\t{\mathfrak t}

\def\v{\mathfrak v}

\def\S{\mathsf S}
\def\T{\mathsf T}
\def\U{\mathsf U}

\def\tlam{\t^\blam}
\def\Tmu{\T^\bmu}
\def\tmu{\t^\bmu}
\def\tnu{\t^\bnu}

\def\Std(#1){{\mathcal T}^{\text{Std}}(#1)}
\def\SStd(#1,#2){{\mathcal T}^{\text{SStd}}_{#2}(#1)}

\DeclareMathAlphabet{\mathpzc}{OT1}{pzc}{m}{it}

\DeclareMathOperator{\Ind}{Ind}
\DeclareMathOperator{\iInd}{\text{$i$}-Ind}

\DeclareMathOperator{\Res}{Res}
\DeclareMathOperator{\Shape}{Shape}

\begin{document}
\bibliographystyle{andrew}
\title{A Specht filtration of an induced Specht module}
\author{Andrew Mathas}
\address{School of Mathematics and Statistics F07, University of
Sydney, NSW 2006, Australia.}
\email{a.mathas@usyd.edu.au}
\thanks{This research was supported, in part, by the Australian
Research Council}

\dedicatory{ To John Cannon and Derek Holt on the occasions of their
significant birthdays,\\ in recognition of their distinguished
contributions to mathematics.}

\begin{abstract}
Let $\H_n$ be a (degenerate or non-degenerate) Hecke algebra of type
$G(\ell,1,n)$, defined over a commutative ring $R$ with one, and let
$S(\bmu)$ be a Specht module for $\H_n$. This paper shows that the
induced Specht module $S(\bmu)\otimes_{\H_n}\H_{n+1}$ has an explicit Specht
filtration.
\end{abstract}

\maketitle

\makeatletter
\def\ps@firstpage{
  \def\@oddhead{{\small \textit{J. Algebra}, \textbf{322} (2009), 893--902.}\hss}
}
\makeatother

\section{Introduction}
The Ariki-Koike algebras, and their rational degenerations, are
interesting algebras which appear naturally in the
representation theory of affine Hecke algebras, quantum groups,
symmetric groups and  general linear groups; see
\cite{M:cyclosurv,Klesh:book}
for details. They include as special cases the group algebras of the
Coxeter groups of type $A$ (the symmetric groups) and the Coxeter
groups of type~$B$ (the hyperoctahedral groups).

Let $\H_n$ be an Ariki-Koike algebra, or a degenerate cyclotomic Hecke
algebra, of type $G(\ell,1,n)$, for integers $\ell,n\ge1$. For each
multipartition $\bmu$ of $n$ there is a \textbf{Specht module} $S(\bmu)$,
which is a right $\H_n$-module.  (All of the undefined terms and notation,
here and below, can be found in section~2.) When $\H_n$ is semisimple the
Specht modules give a complete set of pairwise non-isomorphic irreducible
$\H_n$-modules as $\bmu$ runs through the multipartitions of $n$. In
general, the Specht modules are not irreducible however every irreducible
$\H_n$-module arises, in a unique way, as the simple head of some Specht
module.

The Hecke algebra $\H_n$ embeds into $\H_{n+1}$ so there are natural
induction and restriction functors, $\Ind$ and $\Res$, between the
categories of finite dimensional $\H_n$-modules and $\H_{n\pm1}$-modules.
By \cite[Proposition 1.9]{AM:simples}, in the Ariki-Koike case the
restriction of the Specht module $S(\bmu)$ to $\H_{n-1}$ has a Specht
filtration of the form
\begin{equation}\label{E:restriction}
0=R_0\subset R_1\subset\dots\subset R_r=\Res S(\bmu),
\end{equation}
such that $R_j/R_{j-1}\cong S(\bmu-\rho_j)$, where
$\rho_1>\rho_2>\dots>\rho_r$ are the removable nodes of~$\bmu$. Consequently,
if $\H_{n+1}$ is semisimple then by Frobenius reciprocity
$$\Ind S(\bmu)\cong S(\bmu\cup \alpha_1)\oplus\dots\oplus
	      S(\bmu\cup \alpha_a),$$
where $\alpha_1,\dots,\alpha_a$ are the addable nodes of~$\bmu$.
This note generalizes this result to the case when
$\H_n$ is not necessarily semisimple. More precisely, we prove the
following:

\begin{THEOREM}
Suppose that $\H_n$ is an Ariki-Koike algebra or a degenerate cyclotomic
Hecke algebra of type $G(\ell,1,n)$ and let $\bmu$ be a multipartition of
$n$. Then, as an $\H_{n+1}$-module, the induced module $\Ind S(\bmu)$ has a
filtration
$$0=I_0\subset I_1\subset\dots\subset I_a=\Ind S(\bmu),$$
such that $I_j/I_{j-1}\cong S(\bmu\cup \alpha_j)$, where
$\alpha_1>\alpha_2>\dots>\alpha_a$ are the addable nodes of $\bmu$.
\end{THEOREM}

This result is part of the folklore for the representation theory of
these algebras, however, we have been unable to find a proof of it in
the literature when $\ell>1$. If $\ell=1$ then our Main Theorem is an
old result of James~\cite[\Sect17]{James} in the degenerate case (that
is, for the symmetric group), and it can be deduced from
\cite[Theorem~7.4]{DJ:reps} in the non-degenerate case (the Hecke
algebra of the symmetric group). We prove our Main Theorem by giving
an explicit construction of $\Ind S(\blam)$; see
\autoref{explicit}.  Our argument is similar in spirit to that
originally used by James~\cite{James} for the symmetric groups in that
we identify the induced module as a quotient of the corresponding
permutation module. Our approach, which uses cellular basis
techniques, gives an explicit Specht filtration of the induced module; in
contrast, James' approach is recursive.

Suppose now that $\H_n$ is defined over a field of characteristic
$p\ge0$, or a suitable discrete valuation ring. Then by projecting
onto the blocks of $\H_n$ the induction functor $\Ind$ can be decomposed
as a direct sum of subfunctors
$$\Ind=\bigoplus_{i\in I}\iInd,$$
where $I=\Z/p\Z$, in the degenerate case, and
$I=\set{q^aQ_s|a\in\Z\text{ and }1\le s\le r}$ in the non-degenerate
case. (If the parameters $Q_1,\dots,Q_r$ are all non-zero then, up to
Morita equivalence, it is enough to consider the cases where
$Q_1,\dots,Q_r$ are all powers of $q$ by the main result of
\cite{DM:Morita}. In this case we can take $I=\Z/e\Z$ where $e$ is the
smallest positive integer such that $1+q+\dots+q^{e-1}=0$.) The
functor $\iInd$ is a natural generalization of Robinson's
$i$-induction functor; see \cite[1.11]{AM:simples} and
\cite[\Sect8]{Klesh:book} for the precise definitions.

\begin{Corollary}\label{i-induction}
Suppose that  $\bmu$ is a multipartition of $n$ and $i\in I$. Then $\iInd S(\bmu)$
has a filtration
$$0=I_0\subset I_1\subset\dots\subset I_b=\iInd S(\bmu),$$
such that $I_j/I_{j-1}\cong S(\bmu\cup \alpha_j)$, where
$\alpha_1>\alpha_2>\dots>\alpha_b$ are the addable $i$-nodes of $\bmu$.
\end{Corollary}

\begin{proof}By \cite{LM:AKblocks} and \cite{Brundan:degenCentre}, the
    Specht modules $S(\bmu\cup\alpha)$ and $S(\bmu\cup\beta)$ are in
    the same block if and only if $\alpha$ and $\beta$ have the same
    residue. By the Main Theorem and the definition of the functor $\iInd$,
    the Specht module $S(\bmu\cup\alpha)$ is a subquotient of $\iInd S(\bmu)$
    if and only if $\alpha$ is an $i$-node (\textit{cf.}~\cite[Cor.~1.12]{AM:simples}).
    This implies the result.
\end{proof}

Recently Brundan and Kleshchev~\cite{BK:GradedKL} have shown that $\H_n$ is
naturally $\Z$-graded and Brundan, Kleshchev and
Wang~\cite{BKW:GradedSpecht} have shown that $S(\bmu)$ admits a natural
grading.  There should be a graded analogue of our induction theorem; see
\cite[Remark~4.12]{BKW:GradedSpecht} for a precise conjecture.
Unfortunately, the arguments of this paper do not automatically lift to the
graded setting because it is not clear how to use our results to find a
homogeneous basis of the induced module.

\section{Ariki-Koike algebras}
In order to make this note self-contained, this section quickly
recalls the definitions and results that we need from the literature
and, at the same time, sets our notation. We concentrate on the
non-degenerate case as the degenerate case follows in exactly the same
way, with only minor changes of notation, using the results of
\cite[\Sect6]{AMR}. See the remarks at the end of this section for
more details.

Throughout this note we fix positive integers $\ell$ and $n$ and let
$\Sym_n$ be the symmetric group of degree~$n$. For $1\le i<n$ let
$s_i=(i,i+1)\in\Sym_n$. Then $s_1,\dots,s_{n-1}$ are the standard
Coxeter generators of $\Sym_n$.

Let $R$ be a commutative ring with $1$ and let $q,Q_1,\dots,Q_\ell$ be
elements of $R$ with~$q$ invertible. The Ariki--Koike algebra
$\H_n=\H_{R,\ell,n}(q,Q_1,\dots,Q_\ell)$ is the associative unital
$R$--algebra with generators $T_0,T_1,\dots,T_{n-1}$ and relations
$$\begin{array}{r@{\ }l@{\ }ll}
  (T_0-Q_1)\dots(T_0-Q_\ell) &=&0, \\
  (T_i-q)(T_i+1) &=&0,&\text{for $1\le i\le n-1$,}\\
  T_0T_1T_0T_1&=&T_1T_0T_1T_0,\\
  T_{i+1}T_iT_{i+1}&=&T_iT_{i+1}T_i,&\text{for $1\le i\le n-2$,}\\
  T_iT_j&=&T_jT_i,&\text{for $0\le i<j-1\le n-2$.}
\end{array}$$
Using the relations it follows that there is a unique anti-isomorphism
$*\map{\H_n}\H_n$ such that $T_i^*=T_i$, for $0\le i<n$.

Ariki and Koike~\cite[Theorem~3.10]{AK} showed that $\H_n$ is free as an
$R$-module with basis
$\set{L_1^{a_1}\dots L_n^{a_n}T_w|0\le a_1,\dots,a_n<\ell
         \text{ and } w\in\Sym_n}$
where $L_1=T_0$ and $L_{i+1}=q^{-1}T_iL_iT_i$ for~$i=1,\dots,n-1$, and
$T_w=T_{i_1}\dots T_{i_k}$ if $w=s_{i_1}\dots s_{i_k}\in\Sym_n$ is
a reduced expression (that is, $k$ is minimal).

The Ariki-Koike basis theorem implies that there is a natural
embedding of $\H_n$ in $\H_{n+1}$ and that $\H_{n+1}$ is free as an
$\H_n$-module of rank $\ell(n+1)$.
If $M$ is an $\H_n$-module let
$$\Ind M=M\otimes_{\H_n}\H_{n+1}$$
be the corresponding induced $\H_{n+1}$-module. Note that induction is
an exact functor since $\H_{n+1}$ is free as an $\H_n$-module.

We will need to the following easily proved property of the basis
elements~\cite[2.1]{DJM:cyc}.

\begin{Point}\label{L commute}
  Suppose that $1\le k\le n$, $a\in R$ and $w\in\Sym_k\times\Sym_{n-k}$.
  Then $$(L_1-a)\dots(L_k-a)T_w=T_w(L_1-a)\dots(L_k-a).$$
\end{Point}

The algebra $\H_n$ has another basis which is crucial to this note. In
order to describe it recall that a partition of $n$ is a weakly
decreasing sequence
$\lambda=(\lambda_1\ge\lambda_2\ge\dots)$ of non-negative integers such
that $|\lambda|=\sum_i\lambda_i=n$. A \textbf{multipartition}, or
$\ell$-partition, of $n$ is an ordered $\ell$-tuple
$\blam=(\lambda^{(1)},\dots,\lambda^{(\ell)})$ of partitions such that
$|\blam|=|\lambda^{(1)}|+\dots+|\lambda^{(\ell)}|=n$. Let $\Lambda^+_{\ell,n}$ be the
set of multipartitions of $n$. If $\blam,\bmu\in\Lambda^+_{\ell,n}$ then $\blam$
\textbf{dominates} $\bmu$, and we write~$\blam\gedom\bmu$, if
$$\sum_{t=1}^{s-1}|\lambda^{(t)}|+\sum_{i=1}^k\lambda^{(s)}_i
\ge\sum_{t=1}^{s-1}|\mu^{(t)}|+\sum_{i=1}^k\mu^{(s)}_i,$$
for $1\le s\le \ell$ and for all $k\ge1$. Dominance is a partial order on
$\Lambda^+_{\ell,n}$.

If $\blam$ is a multipartition let
$\Sym_\blam=\Sym_{\lambda^{(1)}}\times\dots\times\Sym_{\lambda^{(\ell)}}$
be the corresponding parabolic subgroup of $\Sym_n$ and set
$a^\blam_s=\sum_{t=1}^{s-1}|\lambda^{(t)}|$, for $1\le s\le\ell$, and
put $a^\blam_{\ell+1}=n-1$. Define $m_\blam=u_\blam^+ x_\blam$ where
 $$u_\blam^+=\prod_{s=2}^\ell\prod_{k=1}^{a^\blam_s}(L_k-Q_s)
 \quad\text{ and }\quad x_\blam=\sum_{w\in\Sym_\blam}T_w.$$
Then $u_\blam^+ x_\blam=m_\blam=x_\blam u_\blam^+$ by \eqref{L commute}.

Let $\blam$ be a multipartition (of $n$).  The \textbf{diagram} of
$\blam$ is the set of nodes
$$[\blam]=\set{(r,c,s)|1\le\lambda^{(s)}_r\le c\text{ and }1\le s\le \ell}.$$
More generally a \textbf{node} is any element of
$\N\times\N\times\{1,\dots,\ell\}$, which we consider as a partially
ordered set where $(r,c,s)\ge(r',c',s')$ if either $s>s'$, or $s=s'$ and
$r<r'$. For the sake of \autoref{i-induction} only, define the
\textbf{residue} of the node $(r,c,s)$ to be $q^{c-r}Q_s$.

An \textbf{addable} node of $\blam$ is any node $\alpha\notin[\blam]$ such
that $[\blam]\cup\{\alpha\}$ is the diagram of some multipartition. Let
$\blam\cup\alpha$ be the multipartition such that
$[\blam\cup\alpha]=[\blam]\cup\{\alpha\}$. Similarly, a \textbf{removable}
node of $\blam$ is a node $\rho\in[\blam]$ such that
$[\blam]-\{\rho\}$ is the diagram of a multipartition; let
$\blam-\rho$ be this multipartition. Note that the set of addable and
removable nodes for $\blam$ are both totally ordered by $>$.

If $X$ is a set then an $X$-valued $\blam$-tableau is a function
$\T\map{[\blam]}X$.  If $\T$ is a $\blam$-tableau then we write
$\Shape(\T)=\blam$.  For convenience we identify
$\T=(\T^{(1)},\dots,\T^{(\ell)})$ with a labeling of the diagram $[\blam]$
by elements of $X$ in the obvious way. Thus, we can talk of the rows,
columns and components of $\T$.

A \textbf{standard $\blam$-tableau} is a map
$\t\map{[\blam]}\{1,2,\dots,n\}$ such that for $s=1,\dots,\ell$ the
entries in each row of $\t^{(s)}$ increase from left to right and the
entries in each column of $\t^{(s)}$ increase from top to bottom.
Let $\Std(\blam)$ be the set of standard $\blam$-tableaux.

Let $\tlam$ be the standard $\blam$-tableau such that the entries in
$\tlam$ increase from left to right along the rows of
$\t^{\lambda^{(1)}}, \dots, \t^{\lambda^{(\ell)}}$ in order.  If $\t$ is a
standard $\blam$-tableau let $d(\t)\in\Sym_n$ be the unique
permutation such that $\t=\tlam d(\t)$. Define
$m_{\s\t}=T_{d(\s)}^*m_\blam T_{d(\t)}$, for $\s,\t\in\Std(\blam)$.
By \cite[Theorem 3.26]{DJM:cyc}, the set
$$\set{m_{\s\t}|\s,\t\in\Std(\blam)\text{ and }\blam\in\Lambda^+_{\ell,n}}$$
is a cellular basis of $\H_n$. Consequently, if $\H_n(\blam)$
is the $R$-module spanned by
$$\set{m_{\s\t}|\s,\t\in\Std(\bmu)\text{ for some
             }\bmu\in\Lambda^+_{\ell,n} \text{ with }\bmu\gdom\blam},$$
then $\H_n(\blam)$ is a two-sided ideal of $\H_n$.

The \textbf{Specht module} $S(\blam)$ is the submodule of
$\H_n/\H_n(\blam)$ generated by $m_\blam+\H_n(\blam)$. It follows from
the general theory of cellular algebras that $S(\blam)$ is free as an
$R$-module with basis $\set{m_\t|\t\in\Std(\blam)}$, where
$m_\t=m_{\tmu\t}+\H_n(\blam)$ for $\t\in\Std(\blam)$.

Let $M$ be an $\H_n$-module. Then $M$ has a \textbf{Specht filtration}
if there exists a filtration $$0=M_0\subset M_1\subset \dots\subset
M_k=M$$ and multipartitions $\blam_1,\dots,\blam_k$ such that
$M_i/M_{i-1}\cong S(\blam_i)$, for $i=1,\dots,k$.

For each multipartition $\bmu\in\Lambda^+_{\ell,n}$ let
$M(\bmu)=m_\bmu\H_n$. The final result that we will need gives an
explicit Specht filtration of $M(\bmu)$. The proof of our Main Theorem
is inspired by this filtration.

Given two tuples $(i,s)$ and $(j,t)$ write $(i,s)\preceq(j,t)$ if
either $s<t$, or $s=t$ and $i\le j$.

\begin{Definition}[\protect{\cite[Definition~4.4]{DJM:cyc}}]
  Suppose that $\blam,\bmu\in\Lambda^+_{\ell,n}$ and let
   $\T\map{[\blam]}\N\times\{1,2,\dots,\ell\}$ be a $\blam$-tableau.
   Then:
   \begin{enumerate}
       \item $\T$ is a tableau of \textbf{type} $\bmu$ if
$\mu^{(s)}_i=\#\set{x\in[\lambda]|\T(x)=(i,s)}$,
for all $i\ge1$ and~$1\le s\le\ell$.
\item  $\T$ is \textbf{semistandard} if the entries
    in each component $\T^{(s)}$, for $1\le s\le \ell$, of $\T$ are:
\begin{enumerate}
  \item weakly increasing from left to right along each row
      (with respect to $\preceq$);
\item strictly increasing from top to bottom down columns; and,
\item $(j,t)$ appears in $\T^{(s)}$ only if $t\ge s$.
\end{enumerate}
\noindent Let $\SStd(\blam,\bmu)$ be the set of semistandard
$\blam$--tableau of type $\bmu$ and let
$\SStd(\Lambda^+_{\ell,n},\bmu)=\bigcup_{\blam\in\Lambda^+_{\ell,n}}\SStd(\blam,\bmu)$ be
the set of all semistandard tableaux of type $\bmu$.
   \end{enumerate}
\end{Definition}

Let $\t$ be a standard $\blam$--tableau. Define $\bmu(\t)$ to be the
tableau obtained from $\t$ by replacing each entry $j$ in $\t$ with $(i,s)$
if $j$ appears in row~$i$ of $\t^{\mu^{(s)}}$. The tableau $\bmu(\t)$ is a
$\blam$--tableau of type $\bmu$; it is not necessarily
semistandard. Finally, if $\S\in\SStd(\blam,\bmu)$ and $\t\in\Std(\blam)$
set
$$m_{\S\t}=\sum_{\substack{\s\in\Std(\blam)\\\bmu(\s)=\S}} m_{\s\t}.$$

\begin{Point}\protect{\cite[Theorem 4.14 and Corollary 4.15]{DJM:cyc}}
\label{sstd basis thm}
Suppose that $\blam,\bmu\in\Lambda^+_{\ell,n}$. Then:
\begin{enumerate}
  \item $M(\bmu)$ is free as an $R$--module with basis
$$\set{m_{\S\t}|\S\in\SStd(\blam,\bmu),\t\in\Std(\blam)
                 \text{ for some }\blam\in\Lambda^+_{\ell,n}}.$$
\item Suppose that $\SStd(\Lambda^+_{\ell,n},\bmu)=\{\S_1,\dots,\S_m\}$ ordered so
  that $i\le j$ whenever $\blam_i\gedom\blam_j$, where $\blam_i=\Shape(\S_i)$.
  Let $M_i$ be the $R$-submodule of $M(\bmu)$ spanned by the elements
  $\set{m_{\S_j\t}|j\le i\text{ and }\in\Std(\blam_j)}$,
Then
$$0=M_0\subset M_1\subset\dots\subset M_m=M(\bmu)$$
is an $\H_n$-module filtration of $M(\bmu)$ and
$M_i/M_{i-1}\cong S(\blam_i)$, for $1\le i\le m$.
\end{enumerate}
\end{Point}

\begin{Remark}
    Very few changes need to be made to the results above in the
    degenerate case. The analogue of the cellular basis $\{m_{\s\t}\}$
    in the degenerate case is constructed in \cite[\Sect6]{AMR}.
    Using this basis of the degenerate Hecke algebra, the construction
    of the Specht filtration of the ideals $M(\bmu)$ follows easily
    using the arguments of \cite[\Sect4]{DJM:cyc};
    \textit{cf.}~\cite[Cor.~6.13]{BK:Higher}. The arguments in the next section,
    modulo minor differences in the meaning of the symbols, applies to
    both the degenerate and non-degenerate cases.
\end{Remark}

\section{Inducing Specht modules}
We are now ready to start proving the Main Theorem. Fix a
multipartition $\bmu\in\Lambda^+_{\ell,n}$. As in \eqref{sstd basis thm}  we let
$\SStd(\Lambda^+_{\ell,n},\bmu)=\{\S_1,\dots,\S_m\}$ be the set of semistandard
tableau of type $\bmu$ ordered so that $i\le j$ whenever
$\blam_i\gedom\blam_j$, where $\blam_i=\Shape(\S_i)$ for $1\le i\le m$. So,
in particular, $\S_m=\Tmu=\bmu(\tmu)$ is the unique semistandard
$\bmu$-tableaux of type~$\bmu$.

Throughout this section we freely identify $\H_n$ with its image under the
natural embedding $\H_n\hookrightarrow\H_{n+1}$. In particular, we
will think of the basis element $m_{\s\t}$ as an element of~$\H_{n+1}$,
for standard $\blam$-tableaux $\s,\t\in\Std(\blam)$ with $\blam\in\Lambda^+_{\ell,n}$. This embedding also identifies $\Ind M(\mu)$ with a submodule of $\H_{n+1}$.

The following simple Lemma contains the idea which drives our proof.

\begin{Lemma}\label{filtration}
  Suppose that $\bmu$ is a multipartition of $n$ and let $\omega$ be the
  lowest addable node of $\bmu$ (that is, $\alpha\ge\omega$ whenever
  $\alpha$ is an addable node of $\bmu$). Then :
\begin{enumerate}
\item $\Ind M(\bmu)=M(\bmu\cup\omega)$.
\item The induced module
$\Ind M(\bmu)$ has a filtration
$$0=N_0\subset N_1\dots\subset N_m=\Ind M(\bmu)$$
such that $N_i/N_{i-1}\cong \Ind S(\blam_i)$, where
$\blam_i=\Shape(\S_i)$ for $1\le i\le m$.
\end{enumerate}
\end{Lemma}

\begin{proof}
  By definition, $m_\bmu=m_{\bmu\cup\omega}$ using the embedding
  $\H_n\hookrightarrow\H_{n+1}$. Therefore,
  $$\Ind M(\bmu)=m_\bmu\H_n\otimes_{\H_n}\H_{n+1}
               =m_\bmu\H_{n+1}=m_{\bmu\cup\omega}\H_{n+1}
               =M(\bmu\cup\omega),$$
proving~(a). As induction is exact, part~(b) follows from part~(a) and
\eqref{sstd basis thm}(b).
\end{proof}

If $\bmu=((n),(0),\dots,(0))$ then $S(\bmu)=M(\bmu)$. The
Main Theorem in this special case is just part~(b) of the Lemma. To
prove the theorem when $\bmu\ne((n),(0),\dots,(0))$ we explicitly
describe the filtration of $\Ind M(\bmu)$ given by the Lemma in terms of
the basis of $M(\bmu\cup\omega)$ from \eqref{sstd basis thm}.

Let $\omega$ be the lowest addable node of $\bmu$. Then
$\omega=(z,1,\ell)$, where $z\ge1$ is minimal such that
$(z,1,\ell)\notin[\bmu]$. Suppose that $\S\in\SStd(\blam,\bmu)$, for
some $\blam\in\Lambda^+_{\ell,n}$, and that~$\beta$ is an addable node of
$\blam$. Let $\S\cup\beta$ be the semistandard
$(\blam\cup\beta)$-tableau given by
$$(\S\cup\beta)(\eta)=\begin{cases}
          \S(\eta),&\text{if }\eta\in[\blam],\\
	  (z,\ell),&\text{if }\eta=\beta.
      \end{cases}$$
Thus $\S\cup\beta$ is the semistandard $(\blam\cup\beta)$-tableau of type
$\bmu\cup\omega$ obtained by adding the node~$\beta$ to~$\S$ with label
$(z,\ell)$.  Let $\SStd(\S,\bmu\cup\omega)$ be the set of semistandard
tableau of type~$\bmu\cup\omega$ obtained in this way from~$\S$ as $\beta$
runs over the addable nodes of $\blam$. It is easy to see that every
semistandard tableau of type~$\bmu\cup\omega$ arises uniquely in this way, so
$$\refstepcounter{equation}
\SStd(\Lambda^+_{\ell,n+1},\bmu\cup\omega)
        =\coprod_{\S\in\SStd(\Lambda^+_{\ell,n},\bmu)}\SStd(\S,\bmu\cup\omega).
\leqno\theequation\label{sstd decomp}
$$
Armed with this notation, observe that if $\S\in\SStd(\blam,\mu)$ then
$m_{\S\tlam}=m_{(\S\cup\beta)\t^{\blam\cup\beta}}$, as an element of
$\H_{n+1}$, where $\beta$ is the lowest addable node of $\blam$.

Suppose that $1\le a\le b<n$. Let $\Sym_{a,b}$ be the symmetric group
on $\{a,a+1,\dots,b\}$ and set $s_{b,a}=(b,b+1)\dots(a,a+1)\in\Sym_n$
and $T_{b,a}=T_{s_{b,a}}=T_b\dots T_a$. For convenience, we set
$T_{b,a}=1$ if $b<a$. The following useful identity
is surely known.

\begin{Lemma}\label{identity}
  Suppose that $1\le a<b\le n$. Then
  $$\Big(\sum_{w\in\Sym_{a,b}} T_w\Big)T_{b,a}
             =T_{b,a}\Big(\sum_{v\in\Sym_{a+1,b+1}}\!\!\!\!T_v\Big).$$
\end{Lemma}

\begin{proof}
  It is easy to check that $\Sym_{a,b}s_{b,a}=s_{b,a}\Sym_{a+1,b+1}$
  and that $s_{b,a}$ is a distinguished
  $(\Sym_{a,b},\Sym_{a+1,b+1})$-double coset representative (in the
  sense of \cite[Prop.~4.4]{M:Ulect}, for example).  Therefore, if
  $w\in\Sym_{a,b}$ and $v=s_{b,a}ws_{b,a}\in\Sym_{a+1,b+1}$ then
  $T_wT_{b,a}=T_{ws_{b,a}}=T_{s_{b,a}v}=T_{b,a}T_v$
  by \cite[Prop.~3.3]{M:Ulect}. This implies the Lemma.
\end{proof}

\begin{Lemma}\label{bump}
  Suppose that $\blam\in\Lambda^+_{\ell,n}$ and $\bnu=\blam\cup\beta$, where
  $\beta=(r,c,e)$ is an addable node of $\blam$.
  Then $T_{n-1,a+1}m_\bnu\in m_\blam\H_{n+1}$, where
  $a=a^\blam_1+\dots+a^\blam_e+\lambda^{(e)}_1+\dots+\lambda^{(e)}_r$.
\end{Lemma}

\begin{proof}
  Let $D_{d,a}=1+T_a+T_{a,a-1}+\dots+T_{a,d}$, where
  $d=a-\lambda^{(e)}_r+1$.  Then~$D_{d,a}$ is the sum of distinguished
  right coset representatives for $\Sym_{d,a}$ in $\Sym_{d,a+1}$.
  Therefore, $x_\blam T_{n-1,a+1}D_{d,a}=T_{n-1,a+1}x_\bnu$ by
  \autoref{identity}.  On the other hand, it follows directly from
  the definitions that
  $u_\bnu^+
     =u^+_\blam(L_{a^\blam_\ell+1}-Q_\ell)\dots(L_{a^\blam_{e+1}+1}-Q_{e+1})$.
  Therefore, writing $m_\blam=x_\blam u^+_\blam$ and using \eqref{L commute} we see
  that
  $$m_\blam\Big(\prod_{s=\ell,\dots,e+1}T_{a^\blam_{s+1},a^\blam_s+1}
            (L_{a^\blam_s+1}-Q_s)\Big) T_{a^\blam_{e+1},a+1}D_{d,a}
  =T_{n-1,a+1}m_\bnu,$$
  where the product on the left-hand side is read in order, from left to
  right, with decreasing values of~$s$. (Recall that, for convenience,
  $a^\blam_{\ell+1}=n-1$ and $T_{n-1,n}=1$.)
\end{proof}

Let $\le$ be the Bruhat order on $\Sym_n$; see, for example,
\cite[p.30]{M:Ulect}.  If $\S$ is a semistandard $\blam$-tableau of
type $\bmu$ let $\dot\S$ be the unique standard $\blam$-tableau such
that $\bmu(\dot\S)=\S$ and $d(\dot\S)\le d(\s)$ whenever
$\s\in\Std(\blam)$ and $\bmu(\s)=\S$. Such a tableau $\dot\S$ exists
by \cite[Lemma~3.9]{JM:cyc-Schaper}.

\begin{Lemma}\label{closure}
  Suppose that $\S\in\SStd(\blam,\bmu)$ and that
  $\U\in\SStd(\S,\bmu\cup\omega)$. Let $\bnu=\Shape(\U)$.
  Then $m_{\U\tnu}\in m_{\S\tlam}\H_{n+1}$.
\end{Lemma}

\begin{proof}
  Definition, $m_{\S\tlam}=\sum_\s m_{\s\tlam}$ where
  $d(\s)$ runs over a set of right $\Sym_\bmu$-coset representatives in the
  double coset $\Sym_\blam d(\dot\S)\Sym_\bmu$. Therefore,
  $m_{\S\tlam}=h_\S T_{d(\dot\S)}^*m_\blam$ for some $h_\S\in\H_q(\Sym_\bmu)$.
  (Explicitly, $h_\S=\sum_d T_d$ where $d$ runs over the set of
  distinguished left coset representatives of
  $\Sym_\bmu\cap d(\dot\S)\Sym_\blam d(\dot\S)^{-1}$ in $\Sym_\bmu$.)

  As in \autoref{bump}, write $\bnu=\blam\cup\beta$, where
  $\beta=(r,c,e)$ and set
  $a=a^\blam_1+\dots+a^\blam_e+\lambda^{(e)}_1+\dots+\lambda^{(e)}_r$.
  Then $\U=\S\cup\beta$. Therefore, $d(\dot\U)=s_{n-1,a+1}^{-1}d(\dot\S)$, so
  that $m_{\U\tnu}=hT_{d(\dot\S)}^*T_{n-1,a+1}m_\bnu$.

  Finally, $T_{n-1,a+1}m_\bnu=m_\blam h_{\bnu,a}$, for some
  $h_{\bnu,a}\in\H_{n+1}$, by \autoref{bump}. Therefore,
  $$m_{\U\tnu}=h_\S T_{d(\dot\S)}^*T_{n-1,a+1}m_\bnu
             =h_\S T_{d(\dot\S)}^*m_\blam h_{\bnu,a}
	     =m_{\S\tlam}h_{\bnu,a}\in m_{\S\tlam}\H_{n+1},$$
  as required.  \end{proof}

We can now make the filtration of \autoref{filtration}(b) explicit.
As a result we will show that we can obtain a basis for the induced
module by adding a node labeled $(z,\ell)$ to the basis elements
of~$M(\bmu)$ in all possible ways.

\begin{Theorem}\label{induced}
  Suppose that $\bmu\in\Lambda^+_{\ell,n}$ and order
  $\SStd(\Lambda^+_{\ell,n},\bmu)=\{\S_1,\dots,\S_m\}$ as above, with
  $\blam_i=\Shape(\S_i)$. Let
  $N_i$ be the $R$-submodule of $M(\bmu\cup\omega)$ spanned by the
  elements
  $$\set{m_{\U\v}|\U\in\SStd(\S_j,\bmu\cup\omega),\v\in\Std(\Shape(\U))
             \text{ for } 1\le j\le i},$$
  for $i=0,1,\dots,m$. Then $N_i$ is an $\H_{n+1}$-submodule of
  $\Ind M(\blam)$ and
  $$\Ind S(\blam_i)\cong N_i/N_{i-1},$$
  for $1\le i\le m$.
\end{Theorem}

\begin{proof}
  By \autoref{filtration}(a), $\Ind M(\bmu)=M(\bmu\cup\omega)$ and
  by \eqref{sstd decomp} the set of elements
  $$\set{m_{\U\v}|\U\in\SStd(\S_j,\bmu\cup\omega),\v\in\Std(\Shape(\U))
             \text{ for } 1\le j\le m}$$
  is precisely the basis of $M(\bmu\cup\omega)$ given
  by \eqref{sstd basis thm}, so $M(\bmu\cup\omega)=N_m$.

  Recall the filtration $0=M_0\subset M_1\subset\dots\subset M_m=M(\blam)$ of
  $M(\blam)$ given in \eqref{sstd basis thm}. By \autoref{filtration}(b), to
  prove the Theorem it is enough to show that $\Ind M_i=N_i$, for $0\le i\le m$.
  By definition, $N_i$ is an $R$-submodule of $M(\bmu\cup\omega)$ but not necessarily an
  $\H_{n+1}$-submodule of $M(\bmu\cup\omega)$. We will show that
  $N_i=N_i\H_{n+1}=\Ind M_i$.  We argue by induction on~$i$. The claim is
  trivially true when $i=0$, so we may assume that $i>0$.

  First observe that if
  $\U\in\SStd(\S_j,\bmu\cup\omega)$ and $\bnu=\Shape(\U)$, where
  $1\le j\le i$, then $m_{\U\v}\in m_{\S_j\t^{\blam_j}}\H_{n+1}\subseteq\Ind M_j$ by \autoref{closure}.
  Therefore, $N_i\H_{n+1}\subseteq\Ind M_i$ since $N_{i-1}=\Ind M_{i-1}$ by
  induction.  Conversely,
  $m_{\S_i\t^{\blam_i}}\in N_i$ and
  $m_{\S_i\t^{\blam_i}}+\Ind M_{i-1}=m_{\S_i\t^{\blam_i}}+N_{i-1}$ generates
  $\Ind M_i/\Ind M_{i-1}$ as an $\H_{n+1}$-module. So, using induction again, it
  follows that $\Ind M_i\subseteq N_i\H_{n+1}$. Hence, we have shown that
  $N_i\H_{n+1}=\Ind M_i$.

  It remains to show that $N_i=N_i\H_{n+1}$ is an $\H_{n+1}$-module.  Before we
  do this, observe that $\Ind M_i/\Ind M_{i-1}\cong\Ind S(\blam_i)$ is a free
  $R$-module of rank $\ell(n+1)\#\Std(\blam_i)$, because $\H_{n+1}$ is a free
  $\H_n$-module of rank $\ell(n+1)$. On the other hand, it is a simple
  combinatorial exercise to show that
  \begin{equation}\label{E:dimensions}
      \ell(n+1)\#\Std(\blam_i)
        = \sum_{\U\in\SStd(\S_i,\bmu\cup\omega)}\#\Std(\Shape(\U)).
  \end{equation}
  This can proved using \autoref{E:restriction}, or alternatively, using the
  ordinary branching rules for the symmetric groups (or for the complex
  reflection groups $(\Z/\ell\Z)\wr\Sym_n$) in characteristic zero.

  We now return to the proof of the theorem. We have that $N_i\subseteq
  N_i\H_{n+1}=\Ind M_i$. Suppose first that $R=K$ is a field.  Then
  $N_i/N_{i-1}$ and $\Ind M_i/\Ind M_{i-1}$ are $K$-vector spaces of the same
  dimension by \autoref{E:dimensions}. Since $N_{i-1}=\Ind M_{i-1}$, by
  induction, it follows that $N_i=N_i\H_{n+1}=\Ind M_i$. Hence, the theorem is
  proved when $R=K$ is a field.

  Finally, suppose that $R$ is an arbitrary commutative ring and let $K$ be the field of
  fractions of $R$.  Let $x$ be any element of $\Ind M_i$. Since
  $\{m_{\U\v}\}$ is a basis of $M(\bmu\cup\omega)$,
  $$x=\sum_{\substack{\U\in\SStd(\S_j,\bmu\cup\omega),\v\in\Std(\Shape(\U))\\1\le j\le m}}
                r_{\U\v}m_{\U\v},$$
  for uniquely determined $r_{\U\v}\in R$. Extending scalars to~$K$, it follows
  from the last paragraph that $r_{\U\v}\ne0$ only if
  $\U\in\SStd(\S_j,\bmu\cup\omega)$, where $1\le j\le i$. Hence,
  $x\in N_i$ with the consequence that $\Ind M_i\subseteq N_i$.  This completes
  the proof.
\end{proof}

For each addable node $\beta$ of $\bmu$ let $N^\beta$ be the submodule of
$M(\bmu\cup\omega)$ spanned by
$$\set{m_{\U\v}|\U\in\SStd(\blam,\bmu\cup\omega),
       \v\in\Std(\blam) \text{ where } \blam\in\Lambda^+_{\ell,n+1}
       \text{ and } \blam\gdom\bmu\cup\beta}+N_{m-1},$$
where $N_{m-1}$ is the submodule of $M(\bmu\cup\omega)$ defined in
\autoref{induced}. Note, in particular, that $N^\alpha=N_{m-1}$.

We can now prove a more explicit version of the Main Theorem of this paper.

\begin{Corollary}\label{explicit}
Suppose that $\bmu$ is a multipartition of $n$ and let
$\alpha_1=\alpha>\dots>\alpha_a=\omega$ be
the addable nodes of $\bmu$. Then
$\Ind S(\bmu)\cong M(\bmu\cup\omega)/N^\alpha$ is a free $R$-module
with basis
$$\set{m_{\U\v}+N^\alpha|\U\in\SStd(\bmu\cup\alpha_j,\bmu\cup\omega),
           \v\in\Std(\bmu\cup\alpha_j), \text{ for }1\le j\le a}.$$
In particular, $\Ind S(\bmu)$ has a filtration
$0=I_0\subset I_1\subset\dots\subset I_a=\Ind S(\bmu)$
such that $I_j/I_{j-1}\cong S(\bmu\cup\alpha_j)$, for $j=1,\dots,a$.
\end{Corollary}

\begin{proof}
    That $\Ind S(\bmu)\cong M(\bmu\cup\omega)/N^\alpha$ is a special
    case of \autoref{induced}. The second claim follows from
    \eqref{sstd basis thm} by setting $I_j=N^{\alpha_{j+1}}/N^\alpha$,
    for $0\le j<a$. To prove that $S(\bmu\cup\alpha_j)\cong I_j/I_{j-1}$, for
    $1\le j\le a$, observe that the bijective map
    $$S(\bmu\cup\alpha_j)\longrightarrow I_j/I_{j-1};
         m_\s\mapsto m_{(\Tmu\cup\alpha_j)\s}+I_{j-1},
	       \qquad\text{for }\s\in\Std(\bmu\cup\alpha_j),$$
    commutes with the action of $\H_{n+1}$. (Here, $\Tmu=\bmu(\tmu)$ is the
    unique semistandard $\bmu$-tableau of type $\bmu$.)
\end{proof}

\begin{Remark}\label{dominance}
    Maintain the notation of \autoref{induced} and define integers
    $a_i$ and multipartitions $\blam_{i,j}$ by
    writing $\{\blam_{i,1}\gdom\dots\gdom\blam_{i,a_i}\}
                   =\set{\Shape(\U)|\U\in\SStd(\S_i,\bmu\cup\omega)}$,
    for $i=1,\dots,m$. \autoref{induced} then implies, just as in the proof of
    \autoref{explicit}, that
    $M(\bmu\cup\omega)$ has a Specht filtration
    $$0\subset I_{1,1}\subset\dots\subset I_{1,a_1}\subset I_{2,1}
       \subset\dots\subset I_{m,a_m}=M(\bmu\cup\omega),$$
    with $I_{i,a}/I_{i,a}^<\cong S(\blam_{i,a})$,
    where $I_{i,a}$ is the submodule of $M(\bmu\cup\omega)$ with basis
    $$\set{m_{\U\v}|\U\in\SStd(\blam_{j,b},\bmu\cup\omega),
        \v\in\Std(\blam_{j,b}) \text{ where } j<i, \text{ or }
               j=i \text{ and }b\le a}$$
    and where $I_{i,a}^<=I_{i,a-1}$ if $a>1$, $I_{i,1}^<=I_{i-1,a_{i-1}}$ if
    $i>1$ and $I_{1,1}^<=0$.

    Fred Goodman has pointed out that this filtration of
    $M(\bmu\cup\omega)$ is, in general, different to that given by
    \eqref{sstd basis thm} because the order in which the Specht modules
    appear does not have to be compatible with the dominance
    ordering--note, however, that the Specht modules in each `layer'
    $N_i/N_{i-1}$ are totally ordered by dominance. For example, suppose that
    $\ell=1$ and let $\mu=(3^2,1)$ so that $\mu\cup\alpha=(4,3,1)$ and
    $\mu\cup\omega=(3^2,1^2)$. Then
    $$\U=\tab(1112,22,34)$$
    is a semistandard $\nu$-tableau of type $\mu\cup\omega$,
    where $\nu=(4,2^2)$. (As $\ell=1$ we can label semistandard
    tableaux with the integers $1,\dots,n$.) However, $\mu\cup\alpha\gdom\nu$ even though
    $\nu\ne\mu\cup\beta$ for any addable node $\beta$ of $\mu$.
\end{Remark}

As induction and restriction are both exact functors the main result
of this note, together with \cite[Prop.~1.9]{AM:simples} (and the
corresponding argument for the degenerate case), shows that the full
subcategory of $\H_n$\textbf{-mod} which consists of modules which
have a Specht filtration is closed under induction and restriction.

\begin{Corollary}\label{Specht filt}
  Suppose that $M$ has a Specht filtration. Then the modules
  $\Res M$ and $\Ind M$ both have Specht filtrations.
\end{Corollary}

In \cite[Theorem~3.6]{M:tilting} and \cite[Theorem 4.6]{BK:degenAK} it
is shown that for each multipartition $\bmu\in\Lambda^+_{\ell,n}$
there exists an indecomposable $\H_n$-module $Y(\bmu)$, a
\textbf{Young module}, such that
$$M(\bmu)\cong Y(\bmu)\oplus\bigoplus_{\blam\gdom\bmu}
             Y(\blam)^{\oplus c_{\blam\bmu}}$$
for some non-negative integers $c_{\blam\bmu}$. Each Young module $Y(\bmu)$
has a Specht filtration. Therefore, by \autoref{Specht filt}, $\Res
Y(\bmu)$ and $\Ind Y(\bmu)$ both have Specht filtrations.


\section*{Acknowledgments}
I thank Fred Goodman and Alexander Kleshchev, first, for (independently) asking the
question that led to this paper and, secondly, for their comments and
suggestions on an earlier draft of the paper. In particular, I thank Fred Goodman
for pointing out a gap in the initial proof of \autoref{induced}


\begin{thebibliography}{10}

\bibitem{AK}
{\sc S.~Ariki and K.~Koike}, {\em A {H}ecke algebra of {$({\bf {Z}}/r{\bf
  {Z}})\wr{\mathfrak {S}}\sb n$} and construction of its irreducible
  representations}, Adv. Math., {\bf 106} (1994), 216--243.

\bibitem{AM:simples}
{\sc S.~Ariki and A.~Mathas}, {\em {The number of simple modules of the Hecke
  algebras of type $G(r,1,n)$}}, Math. Zeit., {\bf 233} (2000), 601--623.

\bibitem{AMR}
{\sc S.~Ariki, A.~Mathas, and H.~Rui}, {\em Cyclotomic {Nazarov--Wenzl
  algebras}}, Nagoya J. Math., {\bf 182} (2006), 47--134, (Special issue in
  honour of George Lusztig).

\bibitem{Brundan:degenCentre}
{\sc J.~Brundan}, {\em Centers of degenerate cyclotomic {H}ecke algebras and
  parabolic category {$\mathcal O$}}, Represent. Theory, {\bf 12} (2008),
  236--259.

\bibitem{BK:GradedKL}
{\sc J.~Brundan and A.~Kleshchev}, {\em Blocks of cyclotomic {H}ecke algebras
  and {K}hovanov-{L}auda algebras}, Invent. Math., {\bf 178} (2009), 451--484.

\bibitem{BKW:GradedSpecht}
{\sc J.~Brundan, A.~Kleshchev, and W.~Wang}, {\em {Graded Specht modules}}, J.
  Reine Angew. Math., {\bf 655} (2011), 61--87,
  \href{http://arxiv.org/abs/0901.0218}{arXiv:0901.0218}.

\bibitem{DJ:reps}
{\sc R.~Dipper and G.~James}, {\em Representations of {Hecke} algebras of
  general linear groups}, Proc. Lond. Math. Soc. (3), {\bf 52} (1986), 20--52.

\bibitem{DJM:cyc}
{\sc R.~Dipper, G.~James, and A.~Mathas}, {\em Cyclotomic $q$--{Schur}
  algebras}, Math.~Z., {\bf 229} (1999), 385--416.

\bibitem{DM:Morita}
{\sc R.~Dipper and A.~Mathas}, {\em Morita equivalences of {Ariki--Koike}
  algebras}, Math. Zeit., {\bf 240} (2002), 579--610.

\bibitem{James}
{\sc G.~D. James}, {\em The representation theory of the symmetric groups},
  SLN, {\bf 682}, Springer--Verlag, New York, 1978.

\bibitem{JM:cyc-Schaper}
{\sc G.~D. James and A.~Mathas}, {\em The {Jantzen} sum formula for cyclotomic
  $q$--{Schur} algebras}, Trans. Amer. Math. Soc., {\bf 352} (2000),
  5381--5404.

\bibitem{Klesh:book}
{\sc A.~S. Kleshchev}, {\em Linear and projective representations of symmetric
  groups}, CUP, 2005.

\bibitem{LM:AKblocks}
{\sc S.~Lyle and A.~Mathas}, {\em Blocks of cyclotomic {H}ecke algebras}, Adv.
  Math., {\bf 216} (2007), 854--878.

\bibitem{M:Ulect}
{\sc A.~Mathas}, {\em {Hecke algebras and Schur algebras of the symmetric
  group}}, Univ. Lecture Notes, {\bf 15}, Amer. Math. Soc., 1999.

\bibitem{M:tilting}
\leavevmode\vrule height 2pt depth -1.6pt width 23pt, {\em Tilting modules for
  cyclotomic {S}chur algebras}, J. Reine Angew. Math., {\bf 562} (2003),
  137--169.

\bibitem{M:cyclosurv}
\leavevmode\vrule height 2pt depth -1.6pt width 23pt, {\em The representation
  theory of the {A}riki-{K}oike and cyclotomic {$q$}-{S}chur algebras}, in
  Representation theory of algebraic groups and quantum groups, Adv. Stud. Pure
  Math., {\bf 40}, Math. Soc. Japan, Tokyo, 2004, 261--320.

\end{thebibliography}
\end{document}